\documentclass[12pt,reqno]{amsart}

\usepackage{amssymb}

\newtheorem{theorem}{Theorem}[section]
\newtheorem{proposition}[theorem]{Proposition}
\newtheorem{lemma}[theorem]{Lemma}

\newtheorem{conjecture}[theorem]{Conjecture}

\theoremstyle{definition}
\newtheorem{definition}[theorem]{Definition}

\numberwithin{equation}{section}

\begin{document}

\baselineskip=15.5pt

\title[Holomorphic Riemannian metric and the fundamental group]{Holomorphic Riemannian metric and 
the fundamental group}

\author[I. Biswas]{Indranil Biswas}

\address{School of Mathematics, Tata Institute of Fundamental
Research, Homi Bhabha Road, Mumbai 400005, India}

\email{indranil@math.tifr.res.in}

\author[S. Dumitrescu]{Sorin Dumitrescu}

\address{Universit\'e C\^ote d'Azur, CNRS, LJAD}

\email{dumitres@unice.fr}

\subjclass[2010]{53B30, 53C50, 53A55}

\keywords{Holomorphic Riemannian metric, algebraic dimension, Killing fields, fundamental group}

\date{}

\begin{abstract} 
We prove that compact complex manifolds bearing a holomorphic Riemannian metric have infinite 
fundamental group.
\end{abstract}

\maketitle

\section{Introduction}

The complex analogue of a (pseudo)-Riemannian metric is a {\it holomorphic Riemannian metric}. Recall that a holomorphic 
Riemannian metric $g$ on a complex manifold $X$ is a holomorphic section of the
vector bundle ${\rm S}^2(T^*X)$ of complex quadratic 
forms on the holomorphic tangent bundle $TX$ which is nondegenerate at every point of $X$ (see Definition~\ref{riem 
metric}).

Given a holomorphic Riemannian metric $g$ on $X$, there is a unique torsion-free
holomorphic affine connection $\nabla$ on the holomorphic tangent bundle $TX$ such that $g$ is parallel with respect to
$\nabla$, or in other words,
$$
\xi \cdot (g(s,\, t))\,=\, g(\nabla_\xi s,\, t)+ g(s,\, \nabla_\xi t)
$$
for all locally defined holomorphic vector fields $\xi$, $s$ and $t$;
this unique connection $\nabla$ is known as the {\it Levi-Civita connection} for $g$.
The curvature tensor of $\nabla$ vanishes identically if 
and only if $g$ is locally isomorphic to the standard flat model $dz_1^2 + \ldots + dz_n^2$ on ${\mathbb C}^n$,
where $n\,=\, \dim_{\mathbb C} X$. More details on the 
geometry of holomorphic Riemannian metrics can be found in \cite{Le,D3,DZ}.

Compact complex manifolds $X$ bearing holomorphic Riemannian metrics are
rather special. First notice that a holomorphic Riemannian metric $g$ on $X$ produces a
holomorphic isomorphism between $TX$ and its dual $T^*X$. In particular, the canonical
line bundle and the anticanonical line bundle
of $X$ are holomorphically isomorphic, which implies that the canonical bundle is of order two (this
implies that the canonical line bundle of a
certain unramified double cover of $X$ is trivial). Moreover, if $X$ is K\"ahler, the classical Chern--Weil theory shows
that the Chern classes with rational coefficients $c_i(X,\mathbb{Q})$ must vanish \cite[pp.~192--193, Theorem~4]{At}. It
now follows, using Yau's theorem
proving Calabi's conjecture~\cite{Ya} (see also~\cite{Be} and ~\cite{IKO}), that $X$ admits a flat K\"ahler metric, and
consequently, $X$ admits a finite unramified cover which is a complex torus. Note that any holomorphic Riemannian metric 
on a complex torus is necessarily translation invariant and, consequently, flat.

An interesting family of compact complex non-K\"ahler manifolds which generalizes complex tori 
consists of those complex manifolds whose holomorphic tangent bundle is holomorphically trivial. 
They are known as {\it parallelizable manifolds}. Any parallelizable manifold is biholomorphic to 
the quotient of a complex Lie group $G$ by a co-compact lattice $\Gamma$ in $G$~\cite{Wa}. A 
parallelizable manifold $G/ \Gamma$ is K\"ahler if and only if $G$ is abelian \cite{Wa}. Any 
nondegenerate complex quadratic form on the Lie algebra of $G$ uniquely defines a right invariant 
holomorphic Riemannian metric on $G$ which descends to a holomorphic Riemannian metric on the 
quotient $G / \Gamma$ of $G$ by a lattice $\Gamma$. In particular, the Killing quadratic form on 
the Lie algebra of a complex semi-simple Lie group $G$, being nondegenerate and invariant under 
the adjoint representation, furnishes a bi-invariant holomorphic Riemannian metric on $G$ and a 
$G$-invariant holomorphic Riemannian metric on any quotient $G / \Gamma$ by a lattice $\Gamma$.

When $G$ is ${\rm SL}(2, \mathbb{C})$, exotic deformations of parallelizable manifolds
${\rm SL}(2, \mathbb{C})/ \Gamma$ bearing holomorphic Riemannian metrics were constructed by Ghys in~\cite{Gh}. Let us
briefly recall Ghys' construction. Choose a uniform lattice $\Gamma$ in ${\rm SL}(2, \mathbb{C})$
as well as a group homomorphism $u\,:\, \Gamma \,\longrightarrow\, {\rm SL}(2, \mathbb{C})$, and consider the embedding 
$$
\Gamma\, \longrightarrow\, {\rm SL}(2, \mathbb{C})\times {\rm SL}(2, \mathbb{C})\, ,\ \
\gamma\,\longmapsto\, (u(\gamma),\, \gamma)\, .
$$
Using this homomorphism, the group
$\Gamma$ acts on ${\rm SL}(2, \mathbb{C})$ via the left and right translations
of ${\rm SL}(2, \mathbb{C})$. More precisely, the action is given by:
$$(\gamma,\,x)\, \longmapsto \, u(\gamma^{-1}) x \gamma \, \in\, {\rm SL}(2,\mathbb{C})$$
for all $(\gamma,\, x) \,\in\, \Gamma \times {\rm SL}(2,\mathbb{C})$. 
It is proved in \cite{Gh} that for $u$ close enough to the trivial homomorphism,
the group $\Gamma$ acts properly and freely on ${\rm SL}(2, \mathbb{C})$ such that the
corresponding quotient $M(u, \Gamma)$ is a compact complex manifold (covered by
${\rm SL}(2, \mathbb{C})$). For a generic homomorphism $u$, these examples do not admit a parallelizable manifold as
a finite cover (see Corollary 5.4 in \cite{Gh} and its proof which shows that these generic examples and their finite covers admit no nontrivial holomorphic vector fields). Since the Killing 
quadratic form on $\text{Lie}({\rm SL}(2, \mathbb{C}))$
is invariant under the adjoint representation, the induced holomorphic Riemannian metric is bi-invariant on 
${\rm SL}(2, \mathbb{C})$ and hence it descends to the quotient $M(u, \Gamma)$. Notice that
a holomorphic Riemannian metric $g$ on $M(u, \Gamma)$ constructed this way is locally isomorphic to the 
complexification of the spherical metric on $S^3$ and it has constant non-zero sectional curvature.
 
The general case of a compact complex threefold $X$ bearing a holomorphic Riemannian metric $g$ shares many features of 
the previous construction of Ghys. In this direction, it was proved in \cite{D3,DZ} that $g$ is necessarily {\it locally 
homogeneous} (see Section \ref{section: geometric structures}), and $X$ admits a finite unramified cover bearing a 
holomorphic Riemannian metric of constant sectional curvature. In view of this we make the following
conjecture:

\begin{conjecture}\label{con1}
Any holomorphic Riemannian metric on a compact complex manifold $X$ is locally homogeneous.
\end{conjecture}

Conjecture \ref{con1} implies that $X$ must have infinite fundamental group (see Section \ref{simply connected}).
 
The main result proved here is the following (see Theorem \ref{main thm}):
 
{\it Every compact complex manifold bearing a holomorphic Riemannian metric has infinite 
fundamental group.}
 
This generalizes Corollary 4.5 and Theorem 4.6 in \cite{BD}, where the same result was proved 
under the extra hypothesis that the algebraic dimension of $X$ is either zero or one.
 
Some parts of the method of proof of Theorem \ref{main thm} generalize to the broader framework of {\it rigid geometric
structures} in the sense of Gromov \cite{DG, Gr} and give the following (see Theorem \ref{rigid}):
 
{\it Let $X$ be a compact complex manifold with trivial canonical bundle and algebraic dimension one. If $X$ admits a 
holomorphic rigid geometric structure, then the fundamental group of $X$ is infinite.}
 
Theorem \ref{rigid} was proved in \cite{BD} under the hypothesis that $X$ is of algebraic dimension zero. It may be
mentioned that for general rigid geometric structures, some hypothesis on the algebraic dimension is needed.
Indeed, projective embeddings of a compact projective Calabi--Yau manifold in some complex projective space are 
holomorphic rigid geometric structures that are not locally homogeneous \cite{DG, Gr}.
 
Nevertheless, we make the following general conjecture which encapsulates the case of holomorphic Riemannian metrics:
 
\begin{conjecture}\label{con2}
Any holomorphic geometric structure of affine type $\phi$ on a 
compact complex manifold with trivial canonical bundle $X$ is locally homogeneous. This implies
that if $\phi$ is rigid, then the fundamental group of $X$ is infinite.
\end{conjecture}

Conjecture \ref{con2} was proved to be true in the following contexts:
\begin{itemize}
\item the complex manifold is K\"ahler, so it is a Calabi--Yau manifold 
\cite{D2};

\item the holomorphic tangent bundle of the manifold is polystable with respect 
to some Gauduchon metric on it \cite{BD};

\item the manifold is Moishezon \cite{BD1};

\item the manifold is a complex torus principal bundle over a compact K\"ahler (Calabi-Yau) manifold \cite{BD1}.
\end{itemize} 

\section{Geometric structures and Killing fields} \label{section: geometric structures}

Let $X$ be a complex manifold of complex dimension $n$.

\begin{definition}\label{riem metric} A {\it holomorphic Riemannian metric}
on $X$ is a holomorphic section
$$
g\, \in\, H^0(X,\, \text{S}^2((TX)^*))
$$
such that for every point $x\, \in\, X$ the quadratic form $g(x)$ on $T_xX$
is nondegenerate. 
\end{definition}

Holomorphic Riemannian metrics and holomorphic affine connections are {\it rigid geometric 
structures} in the sense of Gromov \cite{Gr} (see also \cite{DG}). Let us briefly recall the definition of rigidity in 
the holomorphic category.

For any integer $k \,\geq\,1$, we associate the principal bundle of $k$-frames $R^k(X)\,\longrightarrow\, X$, which is the 
bundle of $k$--jets of local holomorphic coordinates on $X$. The corresponding structural group $D^k$ is the group of 
$k$-jets of local biholomorphisms of $\mathbb{C}^n$ fixing the origin. This $D^k$ is known to be a complex algebraic group.

\begin{definition}
A holomorphic geometric structure $\phi$ of order $k$ on $X$ is a holomorphic $D^k$-equivariant map from $R^k(X)$ to a
complex algebraic manifold $Z$ endowed with an algebraic action of $D^k$. The geometric structure $\phi$ is said to be
of affine type if $Z$ is a complex affine variety.
\end{definition}

Holomorphic tensors are holomorphic geometric structures of affine type of order one, and holomorphic affine connections 
are holomorphic geometric structures of affine type of order two. Holomorphic embeddings in projective spaces, holomorphic 
foliations and holomorphic projective connections are holomorphic geometric structure of non--affine type \cite{DG,Gr}.

A (local) biholomorphism $f$ between two open subsets of $X$ is called a (local) {\it isometry 
(automorphism)} for a geometric structure $\phi$ if the canonical lift of $f$ to $R^k(X)$ 
preserves the fibers of $\phi$.

The associated notion of a (local) infinitesimal symmetry is the following:

\begin{definition} A (local) holomorphic vector field on $Y$ is a (local) {\it Killing field} of a holomorphic
geometric structure $\phi \,:\, R^k(X)\,\longrightarrow\, Z$ if its canonical lift to $R^k(X)$ preserves
the fibers of $\phi$.
\end{definition}

In other words, $Y$ is a Killing field of $\phi$ if and only if its (local) flow preserves $\phi$. The Killing vector 
fields form a Lie algebra with respect to the Lie bracket operation of vector fields.

A classical result in Riemannian geometry shows that $Y$ is a Killing field of a holomorphic 
Riemannian metric $g$ on $X$ if and only if $$TX\,\longrightarrow\, TX\, , \ \ v\, \longmapsto\, 
\nabla_{v} Y$$ is a skew-symmetric section of $\text{End}(TX)$ with respect to $g$, where $\nabla$ 
is the Levi-Civita connection of $g$~\cite{K}.

A holomorphic geometric structure $\phi$ is {\it rigid} of order $l$ in the sense of Gromov if any 
local automorphism of $\phi$ is completely determined by its $l$-jet at any given point (see 
\cite{DG,Gr}).

Holomorphic affine connections are rigid of order one in the sense of Gromov. The rigidity arises from 
the fact that the local biholomorphisms fixing a point and preserving a connection actually 
linearize in exponential coordinates, so they are completely determined by their differential at 
the fixed point. Holomorphic Riemannian metrics, holomorphic projective connections and 
holomorphic conformal structures for dimension at least three are rigid holomorphic geometric 
structures. On the other hand, holomorphic symplectic structures and holomorphic foliations are 
not rigid~\cite{DG}.

Local Killing fields of a holomorphic rigid geometric structure $\phi$ form a locally constant 
sheaf of Lie algebras \cite{DG,Gr}. The typical fiber in that case is a finite dimensional Lie 
algebra called the Killing algebra of $\phi$. The geometric structure $\phi$ is called {\it 
locally homogeneous} if its Killing algebra acts transitively on $X$.

The standard facts about smooth actions of Lie groups preserving an analytic rigid 
geometric structure and a finite volume are adapted to our holomorphic set-up
(compare with \cite[Section 3.5]{Gr} and also with \cite[Section 5]{Fr}).

\begin{lemma}\label{isotropy}
Let $X$ be a compact complex manifold endowed with a holomorphic rigid geometric structure $g$. 
Assume that the automorphism group $G$ of $(X,\,g)$ preserves a smooth volume on $X$ and that $G$ 
is noncompact. Then at any point $x \,\in\, X$, there exists at least one nontrivial local Killing 
field $Y$ of $g$ such that $Y(x)\,=\,0$.
\end{lemma}

\begin{proof}
Let $\phi \,:\, R^k(X) \,\longrightarrow\, Z$ be a holomorphic rigid geometric structure of order $k$. Then there exists 
$l \,\in\, \mathbb{N}$ large enough such that the $l$--jet $$\phi^{(l)} \,:\, R^{k+l}\,
\longrightarrow\, Z^{(l)}$$ of $\phi$
satisfies the condition
that the orbits of the local automorphisms of $\phi$ are the projections on $X$ of the inverse images, through 
$\phi^{(l)}$, of the $D^{k+l}$-orbits in $Z^{(l)}$ \cite{DG, Gr}. Recall that the map $\phi^{(l)}$ is $D^{k+l}$-equivariant.

Since the automorphism group $G$ preserves the finite smooth measure ($X$ is
compact), by Poincar\'e recurrence theorem, for any generic 
point $x \,\in\, X$ there exists an unbounded sequence of elements $g_j \,\in\, G$, $j\, \geq\, 1$,
(meaning a sequence leaving every compact subset in $G$) such that $g_j \cdot x$ converges to $x$. We
lift the $G$--action on $X$ to the bundle $R^{k+l}$ and we consider the orbit $g_j \cdot 
\widehat{x}$ of a lift $\widehat{x}$
of $x$ to $R^{k+l}$. There exists a sequence $\{p_j\}_{j=1}^\infty$ in $D^{k+l}$ such that
$g_j(\widehat{x}) \cdot p_j^{-1}$ converges to $\widehat{x}$. Notice that $\{p_j\}_{j=1}^\infty$ is an unbounded sequence
in $D^{k+l}$, since the lifted $G$-action on $R^{k+l}$ is proper. Using the equivariance property
of $\phi^{(l)}$ we conclude that
$p_j \cdot \phi^{(l)}(\widehat{x})$ converges to $\phi^{(l)}(\widehat{x})$.

The action of the algebraic group $D^{k+l}$ on $Z^{(l)}$ is algebraic. This implies that the $D^{k+l}$-orbits in $Z^{(l)}$ 
are locally closed \cite{Ro}. In particular, this also
remains valid for the orbit $\mathcal O$ of $\phi^{(l)}(\widehat{x})$. Let 
us denote by $I$ the stabilizer of $\phi^{(l)}(\widehat{x})$ in $D^{k+l}$. The orbit $\mathcal O$ with the induced 
topology coming from $Z^{(l)}$ is homeomorphic to the quotient $D^{k+l}/I$. The above observation that $p_j \cdot 
\phi^{(l)}(\widehat{x})$ converges to $\phi^{(l)}(\widehat{x})$ is equivalent to the existence of a sequence $(\eta_j)$ in 
$D^{k+l}$ converging to identity such that $\eta_j \cdot p_j \,\in\, I$. Since $I$ contains an unbounded sequence in $D^{k+l}$ 
and it is an algebraic group --- hence having only finitely many connected components --- it follows that its connected
component $I^0\, \subset\, I$ containing the identity element is a connected complex algebraic subgroup of $D^{k+l}$ of complex dimension at least one. Any
one parameter subgroup in $I^0$ integrates a local Killing field $Y$ vanishing at $x$ (see Corollary 1.6 C in \cite{Gr}). 

We proved that for any recurrent point $x \in X$, there exists at least one nontrivial local 
Killing field Y of $g$ such that $Y(x)\,=\,0$. Equivalently, at any recurrent point, the 
dimension of the local orbit of the Killing algebra is strictly less than the dimension of the 
Killing algebra (which is same on all of $X$). This property is thus true on all of $X$, by 
density of recurrent points in $X$.
\end{proof}

Given a holomorphic Riemannian metric $g$, there is a holomorphic volume form $\omega_g$ 
associated to it, and hence there is an associated volume form on $X$ given by $(\sqrt{-1})^n 
\cdot \omega_g\wedge \overline{\omega_g}$. The automorphism group for $g$ preserves the smooth 
measure on $X$ associated to $(\sqrt{-1})^n \cdot \omega_g\wedge \overline{\omega_g}$.

Recall that the automorphism group of a compact complex simply connected manifold $X$ endowed with 
a holomorphic rigid geometric structure $g$ is known to admit finitely many connected components 
\cite[Section 3.5]{Gr}. Therefore, under these assumptions, the automorphism group of $(X,\,g)$ is 
compact if and only if its connected component containing the identity element is compact.

\section{Algebraic reduction and orbits of Killing fields} \label{section: alg dim}

Recall that the \emph{algebraic dimension} of a compact complex manifold $X$ is the transcendence 
degree of the field of meromorphic functions ${\mathcal M} (X)$ on $X$ over the field of complex 
numbers. The algebraic dimension of a projective manifold coincides with its complex dimension. In 
general, the algebraic dimension of compact complex manifolds of complex dimension $n$ may be less 
than $n$ and in fact takes all integral value between $0$ and $n$. Compact complex manifolds of 
maximal algebraic dimension $n$ are called Moishezon manifolds. Moishezon manifolds are known to 
be bi-meromorphic to projective manifolds \cite{Mo}. More generally there is the following 
classical result called the {\it algebraic reduction} theorem (see \cite{Ue}):

\begin{theorem}[{\cite{Ue}}]\label{thue}
Let $X$ be a compact connected complex manifold of algebraic dimension $a(X)\,=\,d$. Then there
exists a bi-meromorphic modification $$\Psi \,:\, \widetilde{X}\,\longrightarrow\, X\, ,$$ and
a holomorphic map $$t \,:\, \widetilde{X}\,\longrightarrow\, V$$ with connected fibers
onto a $d$-dimensional algebraic manifold $V$, such
that $t^* ({\mathcal M}(V))\,=\, \Psi^*({\mathcal M} (X))$.
\end{theorem} 

In the statement of Theorem \ref{thue} and in the sequel, algebraic manifold means a smooth complex projective manifold.

Let $\pi \,:\, X \,\longrightarrow\, V$ be the meromorphic map given by $t \circ \Psi^{-1}$;
it is called the algebraic reduction of $X$.

\begin{theorem}\label{theorem: algebraic dimension}
Let $X$ be a compact, connected and simply connected complex manifold of complex dimension \(n\) and of algebraic 
dimension $d$. Suppose that $X$ admits a holomorphic rigid geometric structure $g$. Then $H^0(X, \, TX)$ admits an abelian 
subalgebra $A$ acting on $X$ preserving $g$ and
satisfying the condition that each generic fiber of the algebraic reduction of $X$ lies in some
orbit of $A$ (hence the dimension of $A$ is at least $n-d$). Moreover, $A$ is the Lie algebra
of the connected component, containing the identity element, of the automorphism group of the
rigid geometric structure $g'$ which is a juxtaposition of $g$ with a 
maximal family of commuting Killing fields of $g$.
\end{theorem}

Notice that in Theorem \ref{theorem: algebraic dimension} the generic fibers form an analytic 
open dense subset of $X$.

\begin{proof}
By the main theorem in \cite{D1} (see also~\cite{D2}) the Lie algebra of local holomorphic vector 
fields on $X$ preserving $g$ acts on $X$ with generic orbits containing the fibers of the 
algebraic reduction $\pi$ of $X$.

Since $X$ is simply connected, by a result due to Nomizu, \cite{No}, generalized first by Amores, 
\cite{Am}, and then by D'Ambra--Gromov, \cite[p.~73, 5.15]{DG}, local vector fields preserving 
$g$ extend to all of $X$. Thus we get a finite dimensional complex Lie algebra formed by 
holomorphic vector fields $X_i$ preserving $g$, which acts on $X$ with orbits containing the 
generic fibers of the algebraic reduction $\pi$. This finite dimensional complex Lie algebra of 
holomorphic vector fields will be denoted by $\mathcal G$.

Now put together $g$ and a family of global holomorphic vector fields $X_i$ spanning $\mathcal 
G$, to form another rigid holomorphic geometric structure $g'\,=\,(g,\,X_i)$; see \cite{DG} 
(Section 3.5.2 A) for details about the fact that the juxtaposition of a rigid geometric 
structure with another geometric structure is still a rigid geometric structure in the sense of 
Gromov. Considering $g'$ instead of $g$ and repeating the same argument as before, the complex 
Lie algebra $A$ of those holomorphic vector fields preserving $g'$ acts on $X$ with generic 
orbits containing the fibers of the algebraic reduction $\pi$. But preserving $g'$ means preserving
$g$ and commuting 
with the vector fields $X_i$. Hence $A$ coincides with the center of $\mathcal G$. In particular, 
$A$ is a complex abelian Lie algebra acting on $X$ preserving $g$ and with orbits containing the 
generic fibers of the algebraic reduction $\pi$.
\end{proof}

\subsection{Maximal algebraic dimension}

Assume that $X$ is a Moishezon manifold, so the algebraic dimension of
$X$ is $n\,=\, \dim_{\mathbb C}X$.

\begin{proposition}\label{prop1}
If $TX$ admits a holomorphic connection, then $X$ admits a finite unramified 
covering by a compact complex torus.
\end{proposition}

The first step of the proof of Proposition \ref{prop1} is the following:

\begin{lemma}\label{lemma: curve}
Let $X$ be a complex manifold endowed with an affine holomorphic
connection. Then there is no nonconstant holomorphic map from ${\mathbb C}{\mathbb P}^1$ to $X$.
\end{lemma}

\begin{proof}
Let $\nabla$ be a holomorphic connection on $X$. Let
$$
f\, :\, {\mathbb C}{\mathbb P}^1\, \longrightarrow\, X
$$
be a holomorphic map. Consider the pulled back connection $f^*\nabla$ on
$f^*TX$. Since 
$\dim_{\mathbb C} {\mathbb C}{\mathbb P}^1\,=\, 1$, the connection $f^*\nabla$ is flat.
Moreover, ${\mathbb C}{\mathbb P}^1$ being simply connected, $f^*\nabla$ has trivial monodromy,
implying that the holomorphic vector bundle $f^*TX$ is trivial.

Now consider the differential of $f$
$$
df\, :\, T{\mathbb C}{\mathbb P}^1\,\longrightarrow\, f^*TX\, .
$$
There is no nonzero holomorphic homomorphism from $T{\mathbb C}{\mathbb P}^1$ to the trivial holomorphic
line bundle, because $2\,=\, \text{degree}(T{\mathbb C}{\mathbb P}^1)\,> \,
\text{degree}({\mathcal O}_{{\mathbb C}{\mathbb P}^1})\,=\,0$. This implies that $df\,=\, 0$. Therefore,
$f$ is a constant map.
\end{proof}

\begin{proof}[{Proof of Proposition \ref{prop1}}]
Since the Moishezon manifold $X$ does not admit any nonconstant holomorphic map from
${\mathbb C}{\mathbb P}^1$, it is a complex projective manifold \cite[p. 307, Theorem
3.1]{Ca}.

As $TX$ admits a holomorphic connection, and $X$ is a complex projective manifold,
we have $c_i(X, \mathbb{Q})\,=\, 0$ for all $i\, >\, 0$
\cite[p. 192--193, Theorem 4]{At}, where $c_i(X,\mathbb{Q})$ denotes the $i$--th Chern class of $TX$ with
rational coefficients. Therefore, $X$ being complex projective, from Yau's theorem
proving Calabi's
conjecture, \cite{Ya}, it follows that $X$ admits a finite unramified 
covering by a compact complex torus (see also \cite[p. 759, Theorem 1]{Be} and \cite{IKO}).
\end{proof}

\section{Rigid geometric structures and fundamental group} \label{simply connected}

In this section we prove the two main results mentioned in the introduction.

Let us first address the easy case where the geometric structure is taken to be locally 
homogeneous.

\begin{proposition} 
Let $X$ be a compact complex manifold with trivial canonical bundle. If $X$ is endowed with a locally homogeneous 
holomorphic rigid geometric structure $g$, then the fundamental group of $X$ is infinite.
\end{proposition}

\begin{proof}
Assume, by contradiction, that the fundamental group of $X$ is finite. So replacing $X$ by its universal cover
we may assume that $X$ is simply connected. Since $g$ is locally homogeneous, and 
local Killing fields extend to all of $X$ by Nomizu's theorem \cite{No, Am,Gr}, 
it follows that $TX$ is generically spanned by globally defined holomorphic Killing vector 
fields. Let $\{X_1,\, \cdots,\, X_n\}$ a family of linearly independent holomorphic vector fields on $X$ which span $TX$ at the 
generic point. Consider a nontrivial holomorphic section $vol$ of the canonical line bundle, and evaluate it on $X_1\wedge 
\cdots\wedge X_n$ to get a holomorphic function $vol (X_1\wedge 
\cdots\wedge X_n)$. This holomorphic function on $X$ is constant and 
nonzero at the generic point, hence it is nowhere zero. This immediately implies that $\{X_1,\,
\cdots,\, X_n\}$ span $TX$ at
every point in $X$. Consequently, $TX$ 
admits a holomorphic trivialization and hence, by Wang's theorem \cite{Wa}, this $X$ is a quotient of a 
connected complex Lie group by a lattice in it. In particular, the fundamental group of $X$ is
infinite: a contradiction.
\end{proof}

\begin{theorem}\label{main thm}
Let $X$ be a compact complex manifold admitting a holomorphic Riemannian metric $g$. Then the fundamental group of $X$ is 
infinite.
\end{theorem}

\begin{proof}
Assume, by contradiction, that $X$ is endowed with a holomorphic Riemannian metric $g$ and has finite 
fundamental group. Replacing $X$ by its universal cover we can assume that $X$ is simply connected. Denote also by $g$ 
the symmetric bilinear form associated to the quadratic form $g$. The holomorphic tangent bundle $TX$ is endowed with the 
(holomorphic) Levi-Civita connection associated to $g$. If $X$ is a Moishezon manifold, Proposition \ref{prop1} shows that $X$ admits a finite 
unramified cover which is a complex torus: a contradiction (under the assumption that $X$ is Moishezon).

Now consider the situation where the algebraic dimension $d$ of $M$ is strictly less than the complex dimension $n$ of
$M$. Consequently the algebraic reduction of $M$ admits fibers of positive dimension $n-d$ (see Theorem \ref{thue}).

By Theorem \ref{theorem: algebraic dimension}, there exists a finite dimensional abelian Lie 
algebra $A$ lying inside the Lie algebra of global holomorphic vector fields on $X$ such that
\begin{itemize}
\item $A$ preserves $g$, and

\item the generic fibers of the algebraic reduction are contained in the leaves of the foliation 
generated by $A$.
\end{itemize}
Let $X_1,\, X_2, \,\cdots,\, X_{k} \,\in\, H^0(X,\, TX)$ be a basis of $A$;
so we have $k \,\geq\, n-d\, >\, 0$.

For all $i,\,j \,\in\, \{1,\, \cdots,\, k \}$, the functions $g(X_i,\, X_j)$ on $X$ are holomorphic and hence
constant.

Assume first that there exists $X_i \,\in\, A$ as above such that the dual one-form $\omega_i$ defined by
$$\omega_i (v) \,:=\,g(X_i,\, v)$$ vanishes on all tangent vectors tangent to the orbits of $A$ (equivalently,
$g(X_i,\, X_j)\,=\, 0$, for all $j \,\in\, \{1,\, \cdots, \,k \}$). This implies that $\omega_i$ vanishes on the generic
fibers of the algebraic reduction $$\pi \,:\, X \,\longrightarrow\, V\, .$$ We observe that $\Psi^*(\omega_i)$ (see Theorem \ref{thue} for
the map $\Psi$) defines a holomorphic one-form on $\widetilde{X}$. To see that $\Psi^*(\omega_i)$ is a holomorphic one-form on
entire $\widetilde{X}$, first note that the meromorphic map $\Psi$ is holomorphic away from the indeterminacy set $S$ which is of
complex codimension at least two in $\widetilde{X}$. Consequently, $\Psi^*(\omega_i)$ is a holomorphic one-form on the
complement $\widetilde{X}\setminus S$, and hence by Hartog's theorem can be uniquely extended to a holomorphic one-form on
entire $\widetilde{X}$ (for more details the reader is referred to Proposition 1.2 in pages 282--283 of \cite{Ue1}).

Since $\Psi^*(\omega_i)$ vanishes on the generic fibers of $t$ (and hence on all fibers), which are compact and connected, it can be
shown that there is a holomorphic one-form $\widetilde{\omega_i}$ defined on the compact complex projective manifold
$V$ such that $\Psi^*(\omega_i)$ is the pull-back
$t^*(\widetilde{\omega_i})$. To prove this, let us first define the one-form $\widetilde{\omega_i}$ on
the complement $V \setminus S_0$, where $S_0$
is the set of critical values of $t$. Take any $v \,\in\, V \setminus S_0$ and $w \,\in\, T_v V$, and define
$$\widetilde{\omega_i}(v) \cdot w \,:=\, \Psi^*(\omega_i)({\widetilde v}) \cdot \widetilde{w}\, ,$$ where ${\widetilde v} \,\in\,
t^{-1}(v)$ is any element of the (regular) fiber above $v$ and $\widetilde{w} \,\in\, T_{\widetilde v} \widetilde X$ is such that
$dt(\widetilde{v}) \cdot \widetilde{w}\,=\,w$. It should be clarified that this definition does not depend on the choice of $\widetilde w$
because $\Psi^*(\omega_i)$ vanishes on the fiber $t^{-1}(v)$. Moreover this definition does not depend either
on the choice of ${\widetilde v} \,\in\, t^{-1}(v)$, since $t^{-1}(v)$ being a compact and connected manifold, any holomorphic map
from $t^{-1}(v)$ to $(T_vV)^*$ must be a constant one. This construction furnishes a holomorphic one-form $\widetilde{\omega_i}$ on
the complement $V \setminus S_0$ such that $\Psi^*(\omega_i)\,=\,t^*(\widetilde{\omega_i})$ on $t^{-1} (V \setminus S_0)$.

Finally, in order to extend $\widetilde{\omega_i}$ to all of $V$ we make use of the following Lemma 3.3 in \cite{En} (page 57).

\begin{lemma}[{\cite[Lemma 3,3]{En}}]\label{Enoki}
Let $t \,:\, {\widetilde X}\,\longrightarrow\, V$ be a holomorphic mapping with connected fibers between compact complex manifolds
${\widetilde X}$ and
$V$. Let $S_0\,\subset\, V$ be the set of all critical values of $t$. Assume that $V$ is K\"ahler. Then
$$H^0( {\widetilde X}, \, \Omega^1_{\widetilde X}) \cap H^0({\widetilde X} \setminus t^{-1}(S_0), \, t^*\Omega^1_V)
\,\subset\, t^*H^0(V,\, \Omega^1_V)\, .$$
\end{lemma}

Notice that Lemma \ref{Enoki} directly applies to our situation since we already proved that $$\Psi^*(\omega_i)\,\in\, H^0( {\widetilde X}, 
\, \Omega^1_{\widetilde X}) \cap H^0({\widetilde X} \setminus t^{-1}(S_0),\, t^*\Omega^1_V)\, .$$ Hence there exists a holomorphic one-form 
$\widetilde{\omega_i}$ on $V$ such that $\Psi^*(\omega_i)\,=\,t^*(\widetilde{\omega_i})$.

Holomorphic forms on algebraic manifolds
being closed, it follows that $$d \Psi^*(\omega_i)\,=\,d \widetilde{\omega_i}\,=\,0\, .$$ Consequently, we have
$d \omega_i\,=\,0$.

Since $X$ is simply connected, the closed form $\omega_i$ must be exact. This implies that $\omega_i$ vanishes identically, and hence
the vector field $X_i$ vanishes identically: a contradiction.

Thus we are left with the case where $g$ restricted to $A$ is nondegenerate. So assume that
$g$ restricted to $A$ is nondegenerate. This immediately implies that the
vector fields in $A$ do not vanish at any point of $X$. Consequently, the foliation generated by $A$ is 
nonsingular and is of complex dimension $k$. Now Lemma \ref{isotropy} implies that the corresponding 
connected Lie group $G$, meaning the connected component, containing the identity element, of the automorphism group of the holomorphic 
rigid geometric structure $g'\,=\, (g,\, X_1,\, \cdots,\, X_{k})$, is compact. Consequently, $G$ must be isomorphic to a 
compact complex torus $T$ of dimension $n-d$.

The action of $G\,=\,T$ on $X$ is locally free. We will show that this action must be free. For this, assume that an
element $f \,\in 
\,G$ fixes $x_0 \,\in\, X$. Since $G$ is abelian, the differential $df(x_0)$ at $x_0$ acts trivially on $Ax_0$, where
$Ax_0 \,\subset\, T_{x_0}X$ denotes the infinitesimal orbit of $A$ (meaning the evaluation of the Lie algebra of vector fields $A$ at $x_0$). It must preserve its 
$g$-orthogonal part $(Ax_0)^{\perp}$. Moreover, since $f$ preserves each orbit of $G$, the action of $df(x_0)$ must also be trivial on 
$(Ax_0)^{\perp}$. Recall that $g$ restricted to $A$ is nondegenerate, which implies that $$Ax_0 \oplus (Ax_0)^{\perp} \,=\,T_{x_0}X\, .$$
It now follows that $df(x_0)$ is the identity map, and so $f$ must be trivial: it is the identity element in $G$.

The action of $G$ being free, $X$ is a holomorphic principal $G$--bundle over the compact complex manifold $N\,:=\,
X/G$. Since $X$ is simply connected and $G$ is connected, it follows that $N$ is simply connected.
The action of $G$ on $X$ preserves $A$ and $g$, hence the 
restriction of $g$ to $A^{\perp}$ defines a transverse holomorphic Riemannian metric transverse to the foliation defined by the 
$G$-action. This transverse holomorphic Riemannian metric descends to a holomorphic Riemannian metric on $N$. But the 
complex dimension of $N$ is strictly less then the complex dimension of $X$. Now an induction on the complex 
dimension of $X$ finishes the proof; note that the only compact Riemann surfaces admitting holomorphic Riemannian metrics are elliptic curves
and they have infinite fundamental group.
\end{proof}

Recall that it was proved in \cite[Proposition 4.4]{BD} that {\it compact complex manifolds with 
trivial canonical bundle and algebraic dimension zero admitting holomorphic rigid geometric 
structures have infinite fundamental group.}

We prove here the following:

\begin{theorem} \label{rigid}
Let $X$ be a compact complex manifold with trivial canonical bundle and its algebraic dimension is one. If $X$ admits a holomorphic
rigid geometric structure, then the fundamental group of $X$ is infinite.
\end{theorem}
 
\begin{proof}
Let $X$ be a compact complex manifold bearing a holomorphic rigid geometric structure $g$. Assume, by 
contradiction, that the fundamental group of $X$ is finite. Replacing $X$ by its universal cover we assume 
that $X$ is simply connected.

We now use Theorem \ref{theorem: algebraic dimension} to get an abelian subalgebra $A$ of $H^0(X,\, TX)$ which acts on $X$ 
preserving $g$ and satisfying the condition that each generic fiber of the algebraic reduction of $X$ lies in some
$A$-orbit (hence the dimension 
of $A$ is at least $n-1$). Choose elements $X_1,\, \cdots,\, X_{n-1}$ of $A$ which span, at the generic point $x \,\in
\, X$, the tangent space of the fiber $\pi^{-1}(\pi(x))$ of the algebraic reduction $\pi$ of $X$.

Consider a nontrivial holomorphic section $vol$ of the canonical bundle of $X$. Then the holomorphic one-form $\omega$ on $X$ defined 
by $v\, \longmapsto\,vol(X_1\wedge \ldots \wedge X_{n-1}\wedge v)$ vanishes on the fibers of the algebraic reduction $\pi$. As in the proof 
of Theorem \ref{main thm}, the form $\Psi^*(\omega)$ (see Theorem \ref{thue} for $\Psi$) descends to the complex projective manifold 
$V$, the base manifold of the algebraic reduction (once again, extension across singular fibers is not an
issue in view of Lemma \ref{Enoki}). In particular,
the form $\omega$ is closed. This implies that the fundamental group of $X$ is infinite.
\end{proof}

\section*{Acknowledgements}

We thank Ben McKay for interesting discussions on the subject. I.B. is partially supported by a 
J. C. Bose Fellowship.



\begin{thebibliography}{ZZZZ}

\bibitem[Am]{Am} A. M. Amores, Vector fields of a finite type $G$-structure, \textit{Jour. Diff. 
Geom.} \textbf{14} (1980), 1--6.

\bibitem[At]{At} M. F. Atiyah, Complex analytic connections in fibre bundles, \textit{Trans. 
Amer. Math. Soc.} \textbf{85} (1957), 181--207.

\bibitem[Be]{Be} A. Beauville, Vari\'et\'es k\"ahleriennes dont la premi\`ere classe de Chern 
est nulle, \textit{Jour. Diff. Geom.} \textbf{18} (1983), 755--782.

\bibitem[BD1]{BD} I. Biswas and S. Dumitrescu, Holomorphic Affine Connections on Non-K\"ahler 
manifolds, \textit{Internat. J. Math.} {\bf 27} (2016).

\bibitem[BD2]{BD1} I. Biswas and S. Dumitrescu, Fujiki class $\mathcal C$ and holomorphic geometric structures, arxiv.org/abs/1805.11951.

\bibitem[Ca]{Ca} P. Cascini, Rational curves on complex manifolds, \textit{Milan Jour. Math.} 
\textbf{81} (2013), 291--315.

\bibitem[DG]{DG} G. D'Ambra and M. Gromov, \textit{Lectures on transformations groups: geometry 
and dynamics}, Surveys in Differential Geometry, Cambridge MA, (1991).

\bibitem[Du1]{D1} S. Dumitrescu, Meromorphic almost rigid geometric structures, 
\textit{Geometry, Rigidity and Group Actions}, Editors Benson Farb and David Fisher, Chicago 
Lectures in Mathematics Series, 32--58.

\bibitem[Du2]{D2} S. Dumitrescu, Structures g\'eom\'etriques holomorphes sur les vari\'et\'es 
complexes compactes, \textit{Ann. Scient. \'Ec. Norm. Sup} \textbf{34} (2001), 557--571.

\bibitem[Du3]{D3} S. Dumitrescu, Homog\'en\'eit\'e locale pour les m\'etriques riemanniennes 
holomorphes en dimension $3$, \textit{Ann. Inst. Fourier} \textbf{57} (2007), 739--773.

\bibitem[DZ]{DZ} S. Dumitrescu and A. Zeghib, Global rigidity of holomorphic Riemannian metrics 
on compact complex 3-manifolds, \textit{Math. Ann.} \textbf{345} (2009), 53--81.

\bibitem[En]{En} I. Enoki, Generalizations of Albanese mappings for non-K\"ahler manifolds, 
\textit{Geometry and analysison complex manifolds}, Editors Mabuchi and al., World Scientific, 
Singapore, (1994), 201--217.

\bibitem[Fr]{Fr} C. Frances, Variations on Gromov's open-dense orbit theorem, arxiv.org/abs/1605.05755.

\bibitem[Gh]{Gh} E. Ghys, D\'eformations des structures complexes sur les espaces homog\`enes de 
$SL(2, \mathbb{C})$, \textit{Jour. Reine Angew. Math.} \textbf{468} (1995), 113--138.

\bibitem[Gr]{Gr} M. Gromov, Rigid Transfomations Groups, Editors D. Bernard and Y. Choquet 
Bruhat, \textit{G\'eom\'etrie Diff\'erentielle}, \textbf{33}, Hermann, (1988), 65--139.

\bibitem[IKO]{IKO} M. Inoue, S. Kobayashi and T. Ochiai, Holomorphic affine connections on 
compact complex surfaces, \textit{Jour. Fac. Sci. Univ. Tokyo} \textbf{27} (1980), 247--264.

\bibitem[Ko]{K} S. Kobayashi, {\it Transformation groups in differential geometry}, 
\textit{Springer}, (1995).

\bibitem[Le]{Le} C. LeBrun, H-spaces with cosmological constant, \textit{Proc. Royal Soc. 
London}, Ser. A, \textbf{380} (1982) 171--185.

\bibitem[Mo]{Mo} B. Moishezon, On $n$ dimensional compact varieties with $n$ independent 
meromorphic functions, \textit{Amer. Math. Soc. Transl.} \textbf{63} (1967), 51--77.

\bibitem[No]{No} K. Nomizu, On local and global existence of Killing vector fields, \textit{Ann. 
of Math.} \textbf{72} (1960), 105--120.

\bibitem[Ro]{Ro} M. Rosenlicht, A remark on quotient spaces, \textit{An. Acad. Bras. Cienc.} 
\textbf{35} (1963), 483--489.

\bibitem[Ue1]{Ue1} K. Ueno, Classification of algebraic varities, I, \textit{Compositio Math.}
\textbf{27} (1973), 277--342.

\bibitem[Ue]{Ue} K. Ueno, {\it Classification theory of algebraic varieties and compact complet 
spaces}, Springer Lectures Notes, \textbf{439}, (1975).

\bibitem[Wa]{Wa} H.-C. Wang, Complex Parallelisable manifolds, \textit{Proc. Amer. Math. Soc.} 
\textbf{5} (1954), 771--776.

\bibitem[Ya]{Ya} S.-T. Yau, On the Ricci curvature of a compact K\"ahler manifold and the 
complex Monge-Ampère equation. I, {\it Comm. Pure Appl. Math.} {\bf 31} (1978), 339--411.

\end{thebibliography}
\end{document}